\DeclareSymbolFont{AMSb}{U}{msb}{m}{n}
\documentclass[reqno,centertags,12pt,a4paper,noamsfonts]{amsart}
\pdfoutput=1
\usepackage[svgnames]{xcolor}
\usepackage[english]{babel}
\usepackage[babel=true,expansion=alltext,protrusion=alltext-nott,final]{microtype}
\usepackage[margin={3cm},marginparwidth={2.5cm}]{geometry}
\usepackage[utf8]{inputenc}
\usepackage[T1]{fontenc}
\usepackage{libertine}
\usepackage[libertine]{newtxmath}
\useosf
\usepackage[varqu,varl]{inconsolata}
\usepackage{textcomp}
\usepackage{amsthm}
\usepackage{microtype}
\usepackage[inline]{enumitem}
\numberwithin{equation}{section}
\usepackage[normalem]{ulem}

\newcommand{\mr}[1]{\href{https://mathscinet.ams.org/mathscinet/article?mr=#1}{MR~#1}}

\newcommand{\arxiv}[1]{\href{https://arxiv.org/abs/#1}{arXiv~#1}}
\newcommand{\doi}[2]{\href{https://doi.org/#1}{#2}}
\newcommand{\RR}{\mathbb{R}}

\newcommand{\NN}{\mathbb{N}}

\newcommand{\EE}{\mathcal{E}}

\newcommand{\DD}{\mathcal{D}}
\newcommand{\LL}{\mathcal{L}}

\newcommand{\PP}{\mathcal{P}}

\newcommand{\1}{\mathbf{1}}

\newcommand{\eps}{{\varepsilon}}

\DeclareMathOperator{\support}{supp}
\DeclareMathOperator{\supp}{supp}
\DeclareMathOperator{\id}{id}
\newcommand{\dd}{{\operatorfont{d}}}

\DeclareMathOperator{\interior}{Int}

\DeclareMathOperator{\ccexp}{-exp}

\newcommand{\Cexp}[1]{#1\ccexp}

\def\XXint#1#2#3{{\setbox0=\hbox{$#1{#2#3}{\int}$ }
\vcenter{\hbox{$#2#3$ }}\kern-.6\wd0}}

\theoremstyle{plain}
\newtheorem{theorem}{Theorem}[section]
\newtheorem{proposition}[theorem]{Proposition}
\newtheorem{corollary}[theorem]{Corollary}
\newtheorem{lemma}[theorem]{Lemma}

\newtheorem*{theorem*}{Theorem}

\theoremstyle{definition}

\newtheorem{remark}[theorem]{Remark}

\usepackage{hyperref}
\hypersetup{
	linktoc=all,
	colorlinks=true,
	linkcolor=Green,citecolor=Orange,urlcolor=DarkBlue,
	pdftitle={Partial regularity of optimal transport with Coulomb cost},
	pdfauthor={G. Friesecke, T. Ried}, 
	pdfsubject={MSC 2020: 49Q22, 35B65},
	pdfkeywords={optimal transport, Coulomb cost, partial regularity, epsilon-regularity, density functional theory}} 
\usepackage{ulem}
\usepackage{cancel}
\begin{document}
%%%%%%%%%%%%%%%%%%%%%%%%%%%%%%%%%%%%%%%%%%%%%%%%%%
\title{Partial regularity of optimal transport with Coulomb cost}

\author[G.~Friesecke]{Gero Friesecke}
\address[G.~Friesecke]{Technische Universität München, Department of Mathematics, Boltzmannstraße 3, 85748 Garching, Germany}
\email{gf@ma.tum.de}

\author[T.~Ried]{Tobias Ried}
\address[T.~Ried]{School of Mathematics, Georgia Institute of Technology, Atlanta, GA 30332-0160, United States of America} 
\email{tobias.ried@gatech.edu}
%%%%%%%%%%%%%%%%%%%%%%%%%%%%%%%%%%%%%%%%%%%%%%
\date{June 21, 2026}
\thanks{\emph{Funding information}: DFG TRR109 Project-ID 195170736; NSF DMS-2453121.}
\thanks{\textcopyright 2025 by the authors. Faithful reproduction of this article, in its entirety, by any means is permitted for noncommercial purposes.}
%%%%%%%%%%%%%%%%%%%%%%%%%%%%%%%%%%%%%%%%%%%%%%%%%%
\begin{abstract}
We prove that for two-marginal optimal transport with Coulomb cost on $\RR^d$, the optimal map is a $C^{1,\alpha}$ diffeomorphism outside a closed set of Lebesgue measure zero provided the marginals are $\alpha$-H\"older continuous, bounded, and strictly positive. Excluding a set of measure zero is necessary as optimal maps for the Coulomb cost have long been known to exhibit jump singularities across codimension $1$ surfaces (even for smooth marginals on convex domains).
\end{abstract}
%%%%%%%%%%%%%%%%%%%%%%%%%%%%%%%%%%%%%%%%%%%%%%%%%%
\maketitle
%%%%%%%%%%%%%%%%%%%%%%%%%%%%%%%%%%%%%%%%%%%%%%%%%%
\section{Introduction}
Optimal transport with Coulomb cost is an interesting example of optimal transport with \emph{repulsive costs}, which behaves quite differently from standard optimal transport. It has important applications in quantum chemistry, as its general multi-marginal form
arises as the strongly correlated or low-density limit of density functional theory, with the number of marginals corresponding to the number of electrons \cite{Se99, BDG12, CFK13, CFK18}. For a recent review  see \cite{FGG23}.

When it comes to the question of regularity of optimal transport, the Coulomb cost $|x-y|^{-1}$ behaves in a fundamentally different manner from classical costs like $|x-y|^2$, as solutions exhibit jump singularities on codimension 1 surfaces (even for smooth, strictly positive marginals on convex domains). For instance, the optimal map for the uniform density in the one-dimensional interval $[0,L]$ is given by \cite{CFK13}
$$
    T(x) = \begin{cases} x+\tfrac{L}{2} & \mbox{if }x\le\tfrac{L}{2} \\
    x -\tfrac{L}{2} & \mbox{if }x>\tfrac{L}{2}.
    \end{cases}
$$
Analogously, the solutions for radially symmetric (smooth, positive) marginals in any dimension exhibit a jump across a codimension 1 sphere \cite{CFK13}. It follows that the pioneering regularity results for the quadratic cost by Caffarellli \cite{Ca91} and DePhilippis-Figalli \cite{DF11}, which establish H\"older continuity respectively Sobolev ($W^{1,1}$) regularity of optimal maps, as well as subsequent refinements and generalizations to costs satisfying the Ma-Trudinger-Wang condition \cite{Ca92b,MTW05,FL09,Liu09,TW09,LTW10,LTW15,KMC10,FKMC13}, are not applicable. In particular, the best that one can hope for is partial regularity. 

In this paper we prove that for two-marginal optimal transport with Coulomb cost, partial regularity holds. More precisely, the optimal transport map is a $C^{1,\alpha}$ diffeomorphism  outside a closed set of Lebesgue measure zero, provided 
the marginals have a density that is $\alpha$-H\"older continuous, bounded, and strictly positive. The proof proceeds by adapting the recent variational approach to partial regularity for smooth twisted costs with nondegenerate Hessian in \cite{OPR21} to the Coulomb cost.\footnote{Partial regularity for such costs was first proved in \cite{DF15} based on Caffarelli's viscosity approach to Monge-Amp\`ere equations; we do not know whether this approach can be adapted to the singular Coulomb cost.}

The Coulomb cost poses two additional difficulties.
First, it has a singularity along the diagonal $x=y$, thereby failing to provide the regularity properties of the (global) notions of $c$-transforms and $c$-concavity needed in \cite{OPR21}. This is overcome by a careful change of the cost which preserves optimal plans but is no longer globally twisted; see Proposition~\ref{prop:delta-cost} and Section~\ref{sec:2}. 
Second, one needs a local $L^\infty$ bound on the displacement near almost every source point. 
%one needs to show that for almost all points in the support of the source measure, the displacement remains bounded in a small neighborhood. 
Proving this is not straightforward here because the Coulomb cost is repulsive, which has the effect that globally optimal transport wants to move all particles quite far, and -- for physical reasons -- we allow the support of the target measure to be unbounded; see Lemma~\ref{lem:displacement-qualitative}.

\section{Main result} Given probability measures $\mu, \nu$ on $\RR^d$ and the Coulomb cost function $c_0(x,y) = \frac{1}{|y-x|}$ for $x,y\in\RR^d$, we consider the optimal transport problem
\begin{align}\label{eq:ot}
\mbox{minimize }\;
	\mathcal{C} \!\coloneq \!\int_{\RR^d \times \RR^d} \! c_0 \,\mathrm{d}\gamma \;\mbox{ over }\;\gamma\in\Pi(\mu,\nu)
\end{align}
where $\Pi(\mu, \nu)\coloneq \{ \gamma \in\mathcal{P}(\RR^d \times \RR^d): (\pi_1)_{\#}\gamma = \mu, (\pi_2)_{\#}\gamma = \nu\}$ denotes the set of transport plans (couplings) from $\mu$ to $\nu$ and $\mathcal{P}(\RR^n)$ denotes the set of probability measures on $\RR^n$. We denote by $\mathcal{P}\cap L^1(\RR^d)$ the set of absolutely continuous probability measures on $\RR^d$.

We recall the basic existence result for optimizers of $\,\mathcal{C}$:
\begin{proposition}[Theorem 3.1 in \cite{CFK13}]\label{prop:existence}
	Let $\mu, \nu \in \PP\cap L^1(\RR^d)$. Then there exists a unique minimizer $\gamma \in \Pi(\mu, \nu)$ of $\mathcal{C}$. Moreover, there exists a $c_0$-concave function $\psi:\RR^d \to \RR \cup\{-\infty\}$ such that the minimal coupling is of the form $\gamma = (\id \times T)_{\#} \mu$, with transport map $T:\RR^d \to \RR^d$ given by\footnote{The twist condition, i.e.\ $-\nabla_x c_0(x, \cdot)$ is injective for any $x$, implies that the equation $\nabla_x c_0(x,y) + p = 0$ has a unique solution $y$ for any $p$ in the range of $-\nabla_x c_0(x,\cdot)$ and any $x$. This allows one to define $y=\Cexp{c_0}_x(p) \coloneq (-\nabla_x c_0)^{-1}(x, p)$.}
	\begin{align}
		T(x) 
		= \Cexp{c_0}_x(-\nabla\psi(x)) 
		= x + \frac{\nabla \psi(x)}{|\nabla\psi(x)|^{\frac{3}{2}}}, \quad x\in\RR^d.
	\end{align}
\end{proposition}
Note that this result is also interesting for equal marginals $\mu=\nu$; in this case, since $c_0(x,y) = c_0(y,x)$, the optimal coupling $\gamma$ is symmetric, i.e.\ $\gamma(A \times B) = \gamma(B \times A)$ for any Borel subsets $A,B \subseteq \RR^d$. We will, however, not assume $\mu=\nu$ in what follows.

\begin{theorem}\label{thm:partial-regularity}
	Let $\mu,\nu \in \PP\cap L^1(\RR^d)$ be probability measures with Lebesgue densities $\rho_0, \rho_1$ satisfying the following properties: 
	\begin{enumerate}[label=(\roman*)]
		\item $\LL^d(\partial\support\mu) = \LL^d(\partial\support\nu) = 0$;
		\item $\rho_0\in \mathcal{C}^{0,\alpha}(\support\mu)$, $\rho_1 \in \mathcal{C}^{0,\alpha}(\support\nu)$ for some $\alpha\in(0,1)$;
		\item $\rho_0, \rho_1 >0$ on $\support\mu$ and $\support\nu$, respectively;
		\item there exists a constant $M<\infty$ such that $\rho_0, \rho_1 \leq M$.
	\end{enumerate} 
	Then there exist open subsets $X \subset \interior\support\mu$ and $Y\subset \interior\support\nu$ of full measure such that the optimal transport map $T$ from $\mu$ to $\nu$ with respect to the Coulomb cost  $c_0$ is a $\mathcal{C}^{1,\alpha}$-diffeomorphism between $X$ and $Y$.
\end{theorem}

\begin{remark}
Our assumptions are satisfied when $\rho_0=\rho_1=\rho$ is the electron density of the ground state of an atom or molecule, as arising in the application to quantum chemistry. Namely, (i) follows because $\rho$ is analytic away from the nuclei \cite{FHHS04} and hence positive a.e. on $\RR^d$, so that $\partial\, \support \mu$ and $\partial\, \support \nu$ are empty; (ii) and (iv) follow from the fact that ground state wavefunctions have long been known to be globally Lipschitz continuous \cite{Ka57} and exponentially decaying \cite{Ah73}; and (iii) was proved in \cite{FHHS08} in the case of atoms, and is believed to hold also for molecules.
\end{remark}

We note that unlike in \cite{OPR21}, noncompactly supported marginals are allowed here, motivated by the application to quantum chemistry where the support is all of $\RR^d$. This requires some extra work in the proof, more precisely in establishing a  qualitative $L^{\infty}$ bound on the displacement (Lemma~\ref{lem:displacement-qualitative}).

Finally we remark that partial regularity results have previously been obtained for the quadratic cost $|x-y|^2$ in \cite{FK10,GO17,Gol20}, for $p$-type costs away from fixed points in \cite{GK25}, and for smooth twisted costs with non-degenerate Hessian in \cite{DF15,CF17,OPR21}.

\section{Modification of the Coulomb cost that preserves optimal plans} 
Our starting point is the observation that the optimal coupling for the Coulomb cost and related costs stays away from the diagonal. More precisely, let $\delta>0$ and denote by 
\begin{align}\label{eq:diagonal-delta}
	\overline{D}_{\delta} \coloneq \left\{ (x,y)\in\RR^d \times \RR^d: |x-y|\leq\delta \right\}
\end{align}
the ``$\delta$-fattened'' diagonal in $\RR^d \times \RR^d$.

We begin with two preliminary lemmas. For a probability measure $\mu$ on $\RR^d$ we introduce its ``modulus of uniform absolute continuity'' 
\begin{equation} \label{modulus}
  \omega_\mu(\delta) := \sup\{ \mu(A) \, : \, |A|<\delta\}  \;\;\; (\delta >0),
\end{equation}
where $|A|$ denotes the Lebesgue measure of the set $A$.
We will use the following basic non-concentration properties for marginals and couplings which follow from standard measure theoretic arguments:
\begin{lemma} Every function $f\in L^1(\RR^d)$ is uniformly absolutely continuous, i.e. for all $\eps>0$ there exists $\delta>0$ such that
$$
     |A|<\delta \; \Longrightarrow \; \int_A |f|\,  d\LL < \eps.
$$
\end{lemma}
As a consequence, if $\mu$ is absolutely continuous, that is to say there exists $f\in L^1(\RR^d)$ such that $\mu(A)=\int_A f$ for all $A$, then $\omega_\mu(\delta)\to 0$ as $\delta\to 0$.
\begin{proof} 
Because 
$$
     \int_{|f|>C} |f| d\LL \to 0 \mbox{ as }C\to\infty
$$
(e.g. by dominated convergence), there exists $C$ such that 
$\int_{\{|f|>C\}} |f| < \eps/2$. Set $\delta = \tfrac{\eps}{2C}$, then for all $|A|<\delta$
$$
   \int_A |f|d\LL \; = \;  \underbrace{\int_{A\cap \{|f|\le C\}} |f| d\LL}_{\le |A| C < \delta C = \tfrac{\eps}{2}} \;\; + \;\; \underbrace{\int_{A\cap \{|f|>C\}}|f|d\LL}_{<\tfrac{\eps}{2}} < \eps. 
$$
% \textcolor{green}{Not sure if we still need this definition:}
% \begin{align}
% 	\kappa_{\mu}(x_0; r) \coloneq \frac{\mu(B_r(x_0))}{|B_r|}, \quad \text{and} \quad 
% 	\kappa_{\mu}^*(r) \coloneq \sup_{r>0} \sup_{x\in\RR^d} \kappa_{\mu}(x;r).
% \end{align}
\end{proof}

As a consequence, we obtain the following non-concentration property  for the support of a coupling between absolutely continuous measures:
\begin{lemma}\label{lem:non-concentration}
	Let $\mu, \nu$ be absolutely continuous probability measures, and let $\gamma \in \Pi(\mu, \nu)$. Then for any $\varepsilon \in (0,1)$ there exists an $r_0(\varepsilon)>0$ such that for all $(x,y) \in \support\gamma$ there holds
	\begin{align*}
		\gamma(B_r(x)^c \times B_r(y)^c) \geq \varepsilon, \quad \text{whenever} \quad r\leq r_0(\varepsilon).
	\end{align*}
\end{lemma}
\begin{proof}
Let $\varepsilon \in (0,1)$ and fix $(x,y) \in \support\gamma$. Then 
		\begin{align*}
			\gamma\left((B_r(x)^c \times B_r(y)^c)^c\right)
			&\leq \gamma\left((B_r(x)\times \RR^d) \cup (\RR^d \times B_r(y))\right) \\
			&\leq \mu(B_r(x)) + \nu(B_r(y))
			  \\
			&\leq \omega_\mu(|B_r|) + \omega_\nu(|B_r|) 
			\le 1-\varepsilon
		\end{align*}
		if $r\le r_0$ and $r_0$ is chosen small enough, since $\omega_\mu(|B_r|)$, $\omega_\nu(|B_r|)\to 0$ as $r\to 0$. 
\end{proof}

We assume that $c$ is of the form
\begin{equation} \label{cassptns}
   c(x,y) = h(|x-y|) \mbox{ with }h\,:\, [0,\infty)\to\RR\cup\{+\infty\} \mbox{ decreasing}, \; h(r)
   \begin{cases} =\tfrac{1}{r} & \mbox{on }[\delta,\infty) \\
   \le \tfrac{1}{r} & \mbox{otherwise.}
   \end{cases}
\end{equation}
The following lemma extends the well known fact (see \cite{DeP15} for the case of any number of marginals) that the support of optimal plans for the exact Coulomb cost is bounded away from the diagonal. 
\begin{lemma}[Support Lemma]\label{lem:support}
	Let $\mu, \nu$  be absolutely continuous probability measures. Let $r_0$ be the constant from Lemma \ref{lem:non-concentration} with $\varepsilon=\tfrac12$. Then for any $\delta<\frac{r_0}{3}$, for any cost $c$ satisfying \eqref{cassptns}, and any $c$-monotone coupling\footnote{Recall that a set $\Gamma\subseteq \RR^d \times \RR^d$ is $c$-monotone if $c(x,y) + c(w,z) \leq c(w,y) + c(x,z)$ for all $(x,y), (w,z) \in \Gamma$.} $\gamma\in\Pi(\mu,\nu)$, we have $$\gamma(\overline{D}_{\delta})=0.$$ 
\end{lemma}

\begin{proof}%[Proof of Lemma~\ref{lem:support}]
Assume for contradiction that $\gamma(\overline{D}_{\delta})>0$, i.e.\ there exists $(x,y) \in \overline{{D}_{\delta}}\cap \support\gamma$. By Lemma~\ref{lem:non-concentration} with $\varepsilon=\frac{1}{2}$, there exists $(w,z)\in B_{r_0}(x)^c \times B_{r_0}(y)^c \cap \support\gamma$, in particular $|w-x|, |z-y| \geq r_0$. Since $\support\gamma$ is $c$-monotone, there holds 
		\begin{align*}
			h(|x-y|) + h(|w-z|) \leq h(|w-y|) + h(|x-z|).
		\end{align*}

 By the properties of $h$ (see \eqref{cassptns}), the first term is $\ge 1/\delta$ since $|x-y| \leq \delta$, the second term is $\ge 0$, and the third and fourth term are $\le 1/(r_0-\delta)$ since $|w-y| \geq |w-x| - |x-y| \geq r_0 -\delta$ and, analogously, $|x-z|\geq r_0-\delta$.  Hence $\frac{1}{\delta} \leq \frac{2}{r_0-\delta}$, a contradiction since by assumption $\delta < \frac{r_0}{3}$. \qedhere
\end{proof}

%{\color{red}
%	Use that $\nabla\sqrt{\rho} \in L^2$ with maybe a covering argument to get local version of remark? This should also help for densities with non-compact support! 
%}

% We now fix the parameter $\delta=\frac{1}{4}\left(\frac{1}{2M |B_1|}\right)^{\frac{1}{d}}$ from Lemma~\ref{lem:support} and define the modified cost function 
% \begin{align}
% 	c_{\delta}(x,y) \coloneq h_{\delta}(|y-x|), \quad \text{where} \quad 
% 	h_{\delta}(r)\coloneq \begin{cases}
% 		\frac{r^2}{\delta^3} - \frac{3r}{\delta^2} + \frac{3}{\delta}, & 0\leq r < \delta,\\
% 		\frac{1}{r}, & r \geq \delta.
% 	\end{cases}
% \end{align}
% \begin{lemma}[Properties of the cut-off cost $c_{\delta}$]
% 	The cost function $c_{\delta}: \RR^d \times \RR^d \to \RR$ satisfies the following properties: 
% 	\begin{enumerate}[label=(\alph*)]
% 		\item $c_{\delta} \in C^2_b(\RR^d \times \RR^d))$
% 		\item $c_{\delta}$ is globally twisted, i.e.\
% 		\item 
% 	\end{enumerate}
% \end{lemma}
% {\color{red}monotone and below coulomb is enough!}

\begin{proposition}\label{prop:delta-cost}
	Assume that $\mu, \nu \in \mathcal{P}(\RR^d)$ are absolutely continuous probability measures, and let $\gamma_* \in\Pi(\mu, \nu)$ be an optimal coupling for the Coulomb cost $c_0(x,y)={1}/{|x-y|}$. Then there exists a bounded cost $c_\delta\in C^{2,1}(\RR^d\times\RR^d)$ with a.e.\ bounded second derivatives, which satisfies the twist condition on the set $\{(x,y)\in\RR^d\times \RR^d: \, |y-x|>\tfrac{2}{3}\delta\}$ (and therefore on the support of $\gamma_*$), such that $\gamma_*$ is also optimal for $c_{\delta}$.
\end{proposition}
\begin{proof} Since $\gamma_*$ is optimal for the Coulomb cost, it is $c$-monotone for the Coulomb cost, so we can apply Lemma \ref{lem:support}. Let $\delta$ be as in this lemma, then $\gamma_*(\overline{D}_\delta)=0$.
Now take, for instance, $c_\delta(x,y)=h(|x-y|)$ with
$$
   h(r)=\begin{cases}\frac{r^3}{\delta^4} - \frac{2r^2}{\delta^3} +\frac{2}{\delta} & \mbox{for }r\le \delta \\
   \frac{1}{r} & \mbox{otherwise.}
   \end{cases}
$$
\begin{enumerate}[label=\arabic*.]
	\item \textit{Optimality for $c_\delta$.}
		Since $\gamma_*(\overline{D}_\delta)=0$ and $c=c_\delta$ outside $\overline{D}_\delta$, we have 
		\begin{equation} \label{compare1}
		  \inf_\gamma \int c \, d\gamma = \int c \, d\gamma_* = \int c_\delta \, d\gamma_* \ge \inf_\gamma \int c_\delta \, d\gamma.      
		\end{equation}
		On the other hand, let $\gamma_\delta$ be optimal for $c_\delta$. Since $c_\delta$ satisfies \eqref{cassptns}, it follows from Lemma \ref{lem:support} that $\gamma_\delta(\overline{D}_\delta)=0$ and so 
		\begin{equation} \label{compare2}
		  \inf_\gamma \int c_\delta \, d\gamma = \int c_\delta \, d\gamma_\delta = \int c \, d\gamma_\delta \ge \inf_\gamma \int c \, d\gamma.
		\end{equation}
		Hence the last inequality in \eqref{compare1} is an equality, establishing optimality of $\gamma_*$ for $c_\delta$. 

	\item \textit{Regularity of $c_\delta$.}
		Boundedness of $c_\delta$ is obvious. Next, note that the function $h|_{[0,\delta]}$ satisfies $h(\delta)=1/\delta$, $h'(\delta)=-1/\delta^2$ and $h''(\delta)=2/\delta^3$, matching the corresponding derivatives of $h|_{[\delta,\infty)}$. Hence $h$ is $C^{2,\alpha}$ on $(0,\infty)$ for all $\alpha\le 1$. Moreover $z\mapsto $ the last two terms of $h(|z|)$ is smooth on $|z|<\delta$,  and  $z\mapsto$ the first term of $h(|z|)$ is $C^{2,\alpha}$ on $|z|<\delta$ for all $\alpha\le 1$. Thus $z\mapsto h(|z|)$ is $C^{2,\alpha}$ on $\RR^d$ for all $\alpha\le 1$, establishing the asserted regularity of $c_\delta$. 
	
	\item \textit{Boundedness of derivatives.}
		Compute
		\begin{align*}
		  \nabla_x c_{\delta}(x,y) =\begin{cases} h'(|x-y|)\frac{x-y}{|x-y|} &\text{for }y\neq x ,
		  \\
		  0 &\text{for }y=x,
		  \end{cases} \quad \text{and} \quad 
		  \nabla_y c_{\delta}(x,y) = -\nabla_x c_{\delta}(x,y).
		\end{align*}
	Moreover, 
	\begin{align*}
		D_{xx}^2 c_{\delta}(x,y) &= \begin{cases}
 			\frac{h'(|x-y|)}{|x-y|} \mathbb{I} + \left( h''(|x-y|) - \frac{h'(|x-y|)}{|x-y|} \right) \frac{x-y}{|x-y|} \otimes \frac{x-y}{|x-y|}, &\text{for } y\neq x, \\
 			-\frac{4}{\delta^3} \mathbb{I}, & \text{for } y=x,
		 \end{cases}
		\intertext{and}
		D_{yy}^2 c_{\delta}(x,y) &= D_{xx}^2 c_{\delta}(x,y), \quad D_{yx}^2 c_{\delta}(x,y) = D_{xy}^2 c_{\delta}(x,y) = - D_{xx}^2 c_{\delta}(x,y).
	\end{align*}
	Note that 
	\begin{align*}
	\sup_{r>0} |h'(r)|<\infty,
	\quad  
	\sup_{r>0} \left|\frac{h'(r)}{r}\right|<\infty,
	\quad 
	\sup_{r>0}\left|h''(r) - \frac{h'(r)}{r} \right| < \infty,
	\end{align*} 
	and that $\lim_{r\downarrow 0} h'(r) = 0$, $\lim_{r\downarrow 0} \frac{h'(r)}{r} = -\frac{4}{\delta^3}$, $\lim_{r\downarrow 0} \left(h''(r) - \frac{h'(r)}{r} \right) = 0$. 
	\item \textit{Twistedness of $c_\delta$.} It remains to show that $c_\delta$ is twisted on the support of $\gamma_*$, i.e., that $\nabla_x c(x,\cdot)$ is injective on $\{y\in\RR^d\, : \, (x,y)\in \support\gamma\}$. 
		This map is injective on the larger set $\{y\in\RR^d\, : \, |y-x|>\tfrac{2}{3}\delta\}$ thanks to the fact that on $(\tfrac{2}{3}\delta,\infty)$, $h''<0$ and hence $h'$ is strictly decreasing. 
\end{enumerate}
\end{proof}

\begin{remark}\label{rem:inverse-Dxyc}
	From the construction of Proposition~\ref{prop:delta-cost} we also learn that $D^2_{yx} c_{\delta}(x,y)$ is invertible as long as $|x-y| \neq \frac{2}{3} \delta$,  with inverse given by
	\begin{align*}
		D_{yx}^2 c_{\delta}(x,y)^{-1} = -\frac{|x-y|}{h'(|x-y|)} \left( \mathbb{I} - \frac{h''(|x-y|) |x-y| - h'(|x-y|)}{h''(|x-y|)|x-y|} \frac{x-y}{|x-y|}\otimes  \frac{x-y}{|x-y|} \right)
	\end{align*}
	for $x\neq y$.% and $D_{yx}^2 c_{\delta}(x,y)^{-1} = \frac{\delta^3}{4}\mathbb{I}$ for $x=y$. 
\end{remark}

For later reference, let us also collect some properties of the Coulomb cost:
\begin{align*}
	\nabla_x c_0(x, y) = -\frac{x-y}{|x-y|^3}, \quad 
	\nabla_y c_0(x, y) = \frac{x-y}{|x-y|^3}.
\end{align*}

\begin{align*}
	\begin{pmatrix}
		D^2_{xx}c_0(x,y) & D^2_{yx} c_0(x,y) \\
		D^2_{xy}c_0(x,y) & D^2_{yy} c_0(x, y)
	\end{pmatrix} 
	= 
	\frac{1}{|x-y|^3} \begin{pmatrix}
		-\1 + 3e \otimes e & \1 - 3e \otimes e \\
		\1 - 3e \otimes e & -\1 + 3e \otimes e
	\end{pmatrix},
	\quad \text{where } e = \frac{x-y}{|x-y|}. 
\end{align*}
Note in particular that for all $x,y \in \RR^d$, e.g. by Sylvester's determinant rule, 
\begin{align}\label{eq:coulomb-nondegenerate}
	\det D^2_{xy} c_0(x,y) 
   % = |y-x|^{-3d} \det(\1 - 3 e \otimes e) 
    = |y-x|^{-3d} (1-3 |e|^2) = \frac{-2}{|y-x|^{3d}} <0, 
\end{align}
hence $D^2_{xy} c_0(x,y)$ is non-singular for any  $x 
\neq y \in \RR^d$. 

%\begin{corollary}
%	
%\end{corollary}
\section{Properties of the optimal potentials}\label{sec:2}

\begin{lemma}\label{lem:potentials}
		Let $\mu, \nu \in \PP\cap L^1(\RR^d)$. Then there exists a $c_{\delta}$-concave function $\psi_{\delta}:\RR^d \to \RR \cup\{-\infty\}$ such that the ($\mu$-a.e.) unique optimal transport map $T:\RR^d \to \RR^d$ is given by
	\begin{align} \label{eq:maprepr}
		T(x) 
		= \Cexp{c_0}_x(-\nabla\psi_{\delta}(x)) 
		= x + \frac{\nabla \psi_{\delta}(x)}{|\nabla\psi_{\delta}(x)|^{\frac{3}{2}}}, \quad x\in\RR^d.
	\end{align}

\end{lemma}
\begin{proof}[Proof outline]
	Consider the optimal coupling $\gamma \in \Pi(\mu,\nu)$ for the Coulomb cost $c_0$. Then $\gamma$ is also an optimal coupling between the measures $\mu$ and $\nu$ for the cost $c_{\delta}$. Adapting the arguments of \cite{CFK13} (see below for more details), crucially using again that $\gamma(\overline{D_{\delta}}) = 0$ for $\delta$ from Lemma~\ref{lem:support} and $c_{\delta} = c$ on $\support\gamma$, it follows that the optimal coupling is of the form $\gamma = (\id\times T_{\delta})_{\#}\mu$ with $T_{\delta}:\RR^d \to \RR^d$ given by 
\begin{align*}
	T_{\delta}(x) = \Cexp{c_0}_x(-\nabla \psi_{\delta}(x)) = x + \frac{\nabla\psi_{\delta}(x)}{|\nabla\psi_{\delta}(x)|^{\frac{3}{2}}}, \quad x\in\RR^d,
\end{align*}
for some $c_{\delta}$-concave function $\psi_{\delta}:\RR^d \to \RR \cup\{-\infty\}$. Since the optimal transport map for $\mathcal{C}$ is $\mu$-a.e.\ unique, it follows that $T=T_{\delta}$ $\mu$-a.e. 
\end{proof}
In the remainder of this section we give a more detailed proof, explaining how the arguments of \cite{CFK13} can be adapted to obtain the representation \eqref{eq:maprepr}.
\subsection*{More detailed proof of Lemma \ref{lem:potentials}}
Let $\gamma \in \Pi(\mu, \nu)$ be optimal for the cost $c_{\delta}$. Then $\support\gamma$ is $c_{\delta}$-cyclically monotone, i.e.\ there exists a $c_{\delta}$-concave function $\psi_{\delta}: \RR^d \to \RR\cup\{-\infty\}$. In fact, we may define $\psi_{\delta}$ as the $c_{\delta}$-transform of the function $\phi_{\delta}: \RR^d \to \RR\cup\{-\infty\}$ defined via
\begin{align}\label{eq:rueschendorf-2}
	-\phi_{\delta}(y)\coloneq 
 	\inf\Big\{ -c_{\delta}(x_N, y_N) &+ \sum_{i=0}^{N-1}\left( c_{\delta}(x_{i+1},y_i) - c_{\delta}(x_i, y_i)\right): \nonumber \\
    &\qquad (x_0, y_0), \ldots, (x_N, y_N) \in \support\gamma, y=y_N \Big\}
\end{align}
if $y \in \pi_2(\support\gamma)$ and $\phi_{\delta}(y) \coloneq -\infty$ if $y\notin \pi_2(\support\gamma)$.\footnote{Note that since on $\support\gamma$ we have $c_{\delta} = c$ and all pairs $(x_0, y_0), \ldots, (x_N, y_N)$ in \eqref{eq:rueschendorf-2} are contained in $\support\gamma$, it follows that for $y\in \pi_2(\support\gamma)$ 
\begin{align*}
	-\phi_{\delta}(y)=
 	\inf\left\{ -c(x_N, y_N) + \sum_{i=0}^{N-1}\left( c(x_{i+1},y_i) - c(x_i, y_i)\right): \, (x_0, y_0), \ldots, (x_N, y_N) \in \support\gamma, y=y_N \right\}.
\end{align*}
}
Then
\begin{align}\label{eq:psi-phi-cdelta}
	\psi_{\delta}(x) = \inf_{y \in \pi_2(\support\gamma)} \left(c_{\delta}(x,y) - \phi_{\delta}(y)\right) 
	= \inf_{y \in \RR^d} \left(c_{\delta}(x,y) - \phi_{\delta}(y)\right) 
	= \phi_{\delta}^{c_{\delta}}(x)
\end{align}
for all $x\in\RR^d$. Note that for $x\in\pi_1(\support\gamma)$ we have\footnote{For $x\notin\pi_1(\support\gamma)$ we can only use that $c_{\delta} \leq c$ to obtain $\psi_{\delta}(x) \leq \phi_{\delta}^c(x)$.} 
\begin{align*}
	\psi_{\delta}(x) &= \inf_{y \in \pi_2(\support\gamma)} \left(c_{\delta}(x,y) - \phi_{\delta}(y)\right) 
	= \inf_{y \in \pi_2(\support\gamma)} \left(c(x,y) - \phi_{\delta}(y)\right) \\
	&= \inf_{y\in\RR^d}  \left(c(x,y) - \phi_{\delta}(y)\right) 
	= \phi^{c}(x).
\end{align*}
Moreover, by definition of $\phi_{\delta}$ there holds 
\begin{align} \label{eq:rueschendorf} 
	\psi_{\delta}(x) = \inf \Big\{ \sum_{i=0}^N &\left(c(x_{i+1},y_i)-c(x_i,y_i)\right)  : \nonumber \\
	&\quad (x_0,y_0),...,(x_N,y_N)\in \support\gamma, \, x_{N+1}=x, \, N\in\NN\Big\},
\end{align}
which is the R\"uschendorf formula \cite{Rue96} and provides a Kantorovich potential in terms of the support of the optimal plan. 

By \eqref{eq:psi-phi-cdelta} and the definition of $c_{\delta}$ it follows that 
\begin{align}\label{eq:support-bound-upper}
	\psi_{\delta}(x) + \phi_{\delta}(y) \leq c_{\delta}(x,y) \leq c(x,y) \quad \text{on} \quad \RR^d\times\RR^d.
\end{align}
On the other hand, by the standard optimality conditions for optimal plans (see \cite[p.~262]{F24} for a concise account),
\begin{align}\label{eq:support-bound-lower}
	\psi_{\delta}(x) + \phi_{\delta}(y) \geq c_{\delta}(x,y) = c(x,y) \quad \text{on} \quad \support\gamma,
\end{align}
so that 
\begin{align*}
	\psi_{\delta}(x) + \phi_{\delta}(y) = c_{\delta}(x,y) = c(x,y) \quad \text{on} \quad \support\gamma.
\end{align*}

%We will assume from now on that the measure $\mu, \nu$ satisfy the properties of Theorem~\ref{thm:partial-regularity}, in particular $\support\mu, \support\nu$ are compact with non-empty interior.

The following Lemma applies not just to $c_\delta$ but to any cost with bounded and continuous second derivatives. 

\begin{lemma}[$c$-transforms inherit the modulus of semi-concavity of the cost]\label{lem:semiconcave}
	Let $c$ be any cost in $C^2(\RR^d\times\RR^d)$ with $||D^2_x c||_{L^\infty(\RR^d\times\RR^d)}<\infty$. Let $V \subseteq \RR^d$ and let $\psi:\RR^d \to \RR\cup\{-\infty\}$ be the $c$-transform on $V$ of a function $\varphi: V\to \RR\cup\{-\infty\}$, i.e.\ $\psi(x) = \inf_{y\in V} \left(c(x,y)-\varphi(y)\right)$ for $x\in\RR^d$. Then $\psi$ is semi-concave, i.e.\ there exists $K\geq 0$ such that $\psi - \frac{K}{2} |\cdot|^2$ is concave on $\RR^d$.
\end{lemma}

As a consequence of Alexandrov's theorem, see e.g.\ \cite{PR25}, we obtain: 
\begin{corollary}
	Let $\psi$ as in Lemma~\ref{lem:semiconcave} and let $D\coloneq \left\{ x\in\RR^d: \psi(x)>-\infty\right\}$. Then $\psi$ is twice differentiable at Lebesgue-almost every point $x\in D$.
\end{corollary}

%\begin{lemma}\label{lem:semiconcave}
%	The optimal potential $\psi:\support\mu \to \RR$ and its $c_0$-transform $\psi^{c_0}:\support\nu \to \RR$ are locally semi-concave, i.e.\ for any $x\in \interior\support\mu$ there exists an open ball $B(x)\subset \support\mu$ containing $x$ and a real number $K=K(x) \geq 0$ such that $\psi - \frac{K}{2} |\cdot|^2$ is concave on $B(x)$ (respectively for $\psi^{c_0}$).
%\end{lemma}

%\begin{lemma}
%	Let $\varphi, \psi$ be the optimal Kantorovich potentials. Then $\varphi,\psi$ are semi-concave functions. More precisely, for any compact convex subset $K\subset \RR^d$ there exist constants $C_{\varphi, \psi}<\infty$ such that $\varphi - C_{\varphi} |\cdot|^2$ and $\psi - C_{\psi} |\cdot|^2$ are concave functions on $K$. 
%\end{lemma}

\begin{proof}[Proof of Lemma~\ref{lem:semiconcave}]
%	By Proposition~\ref{prop:delta-cost} $\gamma$ is also optimal for the cost $c_{\delta}\in\mathcal{C}^{2,1}(\RR^d\times \RR^d)$ and $c_{\delta}$ is twisted on $\support \gamma$. By standard optimal transport theory the unique optimal coupling $\gamma$ is of the form $\gamma = (\id \times T_{\delta})_{\#}\mu$, where $T_{\delta}(x) = \Cexp{c_{\delta}}_x(-\nabla\psi_{\delta}(x))$ for a $c_{\delta}$-concave function $\psi_{\delta}$ on $\support\mu$. \comment{Here I use that the (second) marginal is compactly supported to avoid checking conditions like (H$\infty$) in Villani's second book, but probably this holds more generally... maybe we can argue more directly via the support?} 
%	By uniqueness, we must have $T = T_{\delta}$, and since $c=c_{\delta}$ on $\support\gamma$, $\nabla\psi = \nabla\psi_{\delta}$. In particular, the potential $\psi:\support\mu\to \RR$ is $c_{\delta}$-concave, that is, there exists a function $\varphi:\support\nu \to \RR$ such that 
%	\begin{align}
%		\psi(x) = \inf_{y \in \support\nu} \left( c_{\delta}(x,y) - \varphi(y) \right)
%	\end{align}
%	for all $x\in\support\mu$. 
	
	Consider the function $\widetilde{\psi} \coloneq \psi - \frac{K}{2} |\cdot|^2$ for some constant $K\geq 0$ to be determined later. 
	
	Let $\lambda \in (0,1)$, $x,x' \in \RR^d$.
	Note that since $c \in \mathcal{C}^{2}(\RR^d \times \RR^d)$ with $\|D_x^2 c\|_{L^{\infty}(\RR^d \times \RR^d)} <\infty$, it follows that for any $z\in \RR^d$,
	\begin{align*}
		&c(\lambda x + (1-\lambda)x', y) - c(z,y) \nonumber \\
		&\quad = \int_0^1 \nabla_x c\left(t(\lambda x + (1-\lambda)x') + (1-t) z, y \right)\dd t \cdot (\lambda x + (1-\lambda)x' - z).
	\end{align*}
	In particular,
	\begin{align}\label{eq:cdelta-expansion-1}
		&\lambda \left(c(\lambda x + (1-\lambda)x', y) - c(x,y)\right) + (1-\lambda) \left(c(\lambda x + (1-\lambda)x', y) - c(x',y)\right) \nonumber \\
		&= -\lambda (1-\lambda) \int_0^1 \left( \nabla_x c\left(x+t(1-\lambda)(x'-x), y \right) - \nabla_x c\left(x'+t\lambda(x-x'), y \right) \right) \dd t \cdot (x -x') \nonumber \\
		&= -\lambda (1-\lambda) (x-x')\cdot \int_0^1\int_0^1 (1-t)\,  D_{x}^2 c\left(X(t,s;\lambda, x,x'),y\right) \dd s \dd t \, (x-x'),
	\end{align}
	where $X(t,s;\lambda, x,x') \coloneq s\left(x+t(1-\lambda)(x'-x)\right) + (1-s)\left(x'+t\lambda(x-x')\right)$.
	
	We now estimate 
	\begin{align*}
		&\widetilde{\psi}(\lambda x + (1-\lambda)x') \\
		&= \inf_{y\in V}\left( c(\lambda x + (1-\lambda)x', y) - \varphi(y) \right) - \frac{K}{2} |\lambda x + (1-\lambda)x'|^2 \\
		&=  \inf_{y\in V}\big[ \lambda \left(c(x,y) - \varphi(y)\right) + (1-\lambda)\left(c(x',y) - \varphi(y) \right) \\
		&\qquad   +\lambda \left(c(\lambda x + (1-\lambda)x', y) - c(x,y)\right) + (1-\lambda) \left(c(\lambda x + (1-\lambda)x', y) - c(x',y)\right) \big] \\
		&\quad
		 - \frac{K}{2} |\lambda x + (1-\lambda)x'|^2 .
	\end{align*}
	Hence, by the definition of $\widetilde{\psi}$ and by means of \eqref{eq:cdelta-expansion-1} we obtain 
	\begin{align*}
		&\widetilde{\psi}(\lambda x + (1-\lambda)x') \\
		&\geq \lambda \widetilde{\psi}(x) + (1-\lambda) \widetilde{\psi}(x') \\
		&\qquad + \lambda \frac{K}{2} |x|^2 + (1-\lambda)\frac{K}{2} |x'|^2 - \frac{K}{2} |\lambda x + (1-\lambda)x'|^2 \\
		&\qquad + \inf_{y\in V}\big[ -\lambda (1-\lambda) (x-x')\cdot \int_0^1\int_0^1 (1-t)\,  D_{x}^2 c\left(X(t,s;\lambda, x,x'),y\right) \dd s \dd t \, (x-x') \big] \\
		&\geq \lambda \widetilde{\psi}(x) + (1-\lambda) \widetilde{\psi}(x') \\
		&\qquad +\lambda(1-\lambda) \frac{K}{2} |x-x'|^2 
		 -\lambda(1-\lambda) \frac{\|D_x^2 c\|_{L^{\infty}(\RR^d \times\RR^d)}}{2} |x-x'|^2.
	\end{align*}
	Therefore, choosing $K = \|D_x^2 c\|_{L^{\infty}(\RR^d \times\RR^d)}$ gives 
	\begin{align*}
		\widetilde{\psi}(\lambda x + (1-\lambda)x') \geq \lambda \widetilde{\psi}(x) + (1-\lambda) \widetilde{\psi}(x'),
	\end{align*}
	i.e.\ $\widetilde{\psi}$ is concave. 
\end{proof}

\begin{lemma} Let $c_{\delta}$ be the cost from Proposition~\ref{prop:delta-cost}, and let $\gamma \in \Pi(\mu,\nu)$ be an optimal coupling for $c_\delta$. Suppose 
	$(x,y)\in\support\gamma$.   Let  $\psi_\delta$ as in \eqref{eq:psi-phi-cdelta},  and suppose that $\psi_{\delta}:\RR^d \to \RR$ is differentiable at $x$. Then $y = x + \frac{\nabla\psi_{\delta}(x)}{|\nabla\psi_{\delta}(x)|^{\frac{3}{2}}}$. 
\end{lemma}
\begin{proof}
    By \eqref{eq:support-bound-upper} and \eqref{eq:support-bound-lower}, $c_\delta(x,y)-\psi_\delta(x)-\phi_\delta(y)$ is minimal on $\support\gamma$, and consequently $\nabla_x c_\delta(x,y)=\nabla\psi_\delta(x)$. Since $|x-y|>\delta$ by Lemma \ref{lem:support}, this equation also holds with $c_\delta$ replaced by $c_0$. Solving for $y$ gives the assertion. 
	% This is a simple adaptation of \cite[Lemma~3.11]{CFK13}, using that $y\mapsto \nabla_x c_{\delta}(x, y)$ is injective on the set $\{y:|y-x| > \delta\}$ and $c_{\delta} = c$ on $\overline{D_{\delta}}^c$. 
\end{proof}

We may therefore define $T_{\delta}: \RR^d \to \RR^d$ via $T_{\delta}(x) \coloneq x+\frac{\nabla\psi_{\delta}(x)}{|\nabla\psi_{\delta}(x)|^{\frac{3}{2}}}$ for the $c_{\delta}$-concave potential $\psi_{\delta}$. Then the map $T_{\delta}$ pushes forward $\mu$ to $\nu$, and there holds $\gamma = (\id\times T_{\delta})_{\#}\mu$. 

Indeed, for any test function $\zeta \in \mathcal{C}_b(\RR^d)$ we have\footnote{Note that the map $z\mapsto \frac{z}{|z|^{3/2}}$ is a Borel map as pointwise limit of continuous functions.}
\begin{align*}
	\int \zeta \,\dd(T_{\delta})_{\#}\mu 
	&= \int \zeta(T_{\delta}(x))\,\mu(\dd x) 
	= \int \zeta(T_{\delta}(x))\,\gamma(\dd x \dd y),
\end{align*}
and using that $y=T_{\delta}(x)$ on $\support \gamma$ it follows that 
\begin{align*}
	\int \zeta(T_{\delta}(x))\,\gamma(\dd x \dd y) 
	= \int \zeta(y)\,\gamma(\dd x \dd y) 
	= \int \zeta \,\dd\nu.
\end{align*}
By $\mu$-a.e. uniqueness of the optimal transport map $T$ for $\mathcal{C}$, it follows that $T_{\delta}=T$ $\mu$-a.e.
\section{Proof of Partial Regularity}
The proof of partial regularity for the optimal transport map with respect to the Coulomb cost essentially consists of two parts: 
\begin{enumerate}[label=\arabic*.]
		\item Around any point in the support of the optimal coupling, the Coulomb cost is well-approximated by a Euclidean cost, i.e.\ by the cost $-x\cdot y$ after a suitable affine change of coordinates. This relies on the support lemma (Lemma~\ref{lem:support}) and the non-degeneracy of Coulomb cost away from the diagonal (see \eqref{eq:coulomb-nondegenerate}). 
	\item Around any point in the support of the optimal coupling, the Euclidean transport energy is small, which together with local closeness of the Coulomb cost function to quadratic cost and regularity of the marginals puts us into the regime where the $\varepsilon$-regularity result of \cite{OPR21} holds. 
	
	The local Euclidean transport energy at scale $R$ of a coupling $\gamma$ around a point $(x_0, y_0)$ is measured by the quantity
	\begin{align}
		\EE_R(\gamma, (x_0, y_0)) \coloneq \frac{1}{R^{d+2}} \int_{\#_R(x_0, y_0)} |y-x|^2 \, \gamma(\dd x\dd y),
	\end{align}
	% where we recall the definition of the infinite cross
	% \begin{align*}
	% 		\#_{R}(x_0, y_0) = (B_{R}(x_0) \times \RR^d) \cup (\RR^d \times B_{R}(y_0)) .
	% \end{align*}
where, for a given point $(x_0, y_0) \in \RR^d \times \RR^d$ and radius $R>0$, we denote by   
\begin{align*}
	\#_{R}(x_0, y_0) \coloneq (B_{R}(x_0) \times \RR^d) \cup (\RR^d \times B_{R}(y_0)) 
\end{align*}
the ``infinite cross'' centered at $(x_0, y_0)$ of width $R$. 
    Note that $\EE_R$ amounts both for the Euclidean transport energy that it costs to move particles from $B_R(x_0)$ and to move particles into $B_R(y_0)$ under the plan $\gamma$. In fact, we first bound the energy of the ``forward'' part on scale $6R$ 
	\begin{align}
		\EE^+_{6R} 
		\coloneq \frac{1}{(6R)^{d+2}} \int_{B_{6R}(x_0)\times \RR^d} |y-x|^2 \, \gamma(\dd x\dd y) 
		= \frac{1}{(6R)^{d+2}} \int_{B_{6R}(x_0)} |T(x)-x|^2 \,\mu(\dd x),
	\end{align}
	and combine this with a precise bound on the local displacement
    % \begin{align*}
    %     (x,y) \in &(B_{3R}(x_0)\times\RR^d) \cap \support{\gamma} \\
    %     &\Rightarrow \quad |x-y| \lesssim R \left( \mathcal{E}_{4R}^+ + \mathcal{D}_{4R}(\mu, \nu; (x_0, y_0) \right)^{\frac{1}{d+2}},
    % \end{align*}
    % where $\mathcal{D}_{4R}(\mu, \nu; (x_0, y_0)$ is introduced in \eqref{eq:data} below. This then implies that 
    % \begin{align*}
    %     (\RR^d \times B_{2R}(y_0)) \cap \support{\gamma} \subseteq B_{4R}(x_0) \times B_{2R}(y_0)
    % \end{align*}
    % (see \cite[Step 1 in the proof of Theorem 1.1]{OPR21}), giving us control of 
    to control the full energy $\mathcal{E}_{2R}$ on the smaller scale $2R$, see \eqref{eq:two-sided-energy-bound}.
	
	A crucial ingredient in the $\varepsilon$-regularity result of \cite{OPR21} is an initial bound on the displacement, which has to be slightly modified so that it applies to the cost function $c_{\delta}$ and to densities with unbounded support.
	
	The Hölder regularity of the marginals, together with their strict positivity on their support, implies that the marginals are locally well-approximated by unit-mass Lebesgue measure, in the sense that 
	\begin{align*}%\label{eq:data}
		\mathcal{D}_R(\mu, \nu; (x_0, y_0)) 
		&\coloneq \frac{1}{R^{d+2}} W_2^2\Big(\mu\lfloor_{B_R(x_0)}, \frac{\mu(B_R(x_0))}{|B_R|} \dd x\lfloor_{B_R(x_0)}\Big) + \Big(\frac{\mu(B_R(x_0))}{|B_R|}-1\Big)^2 \nonumber\\
		&\quad + \frac{1}{R^{d+2}} W_2^2\Big(\nu\lfloor_{B_R(y_0)}, \frac{\nu(B_R(y_0))}{|B_R|} \dd y\lfloor_{B_R(y_0)}\Big) + \Big(\frac{\nu(B_R(y_0))}{|B_R|}-1\Big)^2 
	\end{align*}
	is small for $R$ small enough. Indeed, using that $\mu$ and $\nu$ are bounded away from zero on $B_R(x_0)$ and $B_R(y_0)$, respectively, the Moser construction combined with elliptic regularity based on the Hölder regularity of the densities $\rho_0$ of $\mu$ and $\rho_1$ of $\nu$ implies that 
	\begin{align*}
		\frac{1}{R^{d+2}} W_2^2\Big(\mu\lfloor_{B_R(x_0)}, \frac{\mu(B_R(x_0))}{|B_R|} \dd x\lfloor_{B_R(x_0)}\Big) 
		\lesssim R^{2\alpha} [\rho_0]_{\alpha; B_R(x_0)}^2 
		\intertext{and}
		\frac{1}{R^{d+2}} W_2^2\Big(\nu\lfloor_{B_R(y_0)}, \frac{\nu(B_R(y_0))}{|B_R|} \dd y\lfloor_{B_R(y_0)}\Big) 
		\lesssim R^{2\alpha} [\rho_1]_{\alpha; B_R(y_0)}^2, 
	\end{align*}
	hence $\mathcal{D}_R(\mu, \nu; (x_0, y_0)) \lesssim R^{2\alpha} [\rho_0]_{\alpha; B_R(x_0)}^2 + R^{2\alpha} [\rho_1]_{\alpha; B_R(y_0)}^2$. 
	We refer to \cite[Lemma~A.4]{OPR21}.
\end{enumerate}

Let $T:\support{\mu} \to \support{\nu}$ be the $c_0$-optimal map from $\mu$ to $\nu$, and let $T^*:\support\nu \to \support\mu$ be the $c_0$-optimal map from $\nu$ to $\mu$. Then $T$ and $T^*$ are $\LL^d$-almost everywhere inverses of each other, and are given by
\begin{align*}
	T(x) &= \Cexp{c_0}_x(-\nabla\psi_{\delta}(x)), \quad \text{and} \\
	T^*(y) &= \Cexp{c_0}_y(-\nabla\psi_{\delta}^{c_{\delta}}(y)),
\end{align*}
%\red{By Proposition~\ref{prop:delta-cost}, $T$ and $T^*$ are also $c_{\delta}$-optimal between the respective marginals.} 
Since the potentials $\psi_{\delta}$ and $\psi_{\delta}^{c_{\delta}}$ are locally semi-concave by Lemma~\ref{lem:semiconcave}, they have a second derivative Lebesgue-almost everywhere.

Hence, we can find two open sets $X'\subseteq \interior\support\mu$ and $Y' \subseteq \interior\support\nu$ of full Lebesgue measure such that for all $(x_0, y_0) \in X\times Y$ there holds
    \begin{align}\label{eq:inverse-relation1}
	   & y_0 = T(x_0) \quad \text{and} \quad  x_0 = T^*(y_0), \\
      &T^*(T(x_0)) = x_0 \quad \text{and} \quad T(T^*(y_0)) = y_0, \label{eq:inverse-relation2}, \\
      & \psi_{\delta}, \psi_{\delta}^{c_{\delta}}: \RR^d \to \RR\cup\{-\infty\} \mbox{ and } \mbox{ are twice  differentiable at }x_0 \mbox{ respectively }y_0.
\end{align}
		In particular, this implies that $T$ is differentiable at $x_0$ and $T^*$ is differentiable at $y_0$.

Now set 
\begin{align*}
	X\coloneq X' \cap T^{-1}(Y') \quad \text{and} \quad Y \coloneq Y' \cap T(X').
\end{align*}
Note that since $T_{\#}\mu = \nu$ and by assumption $\mu$ and $\nu$ are absolutely continuous with respect to Lebesgue measure with densities that are bounded away from zero on any compact subset of their supports, we have that $\LL^d(\support\mu\setminus X) = 0$, and similarly $\LL^d(\support \nu \setminus Y) = 0$.

\subsection{A couple of normalizations}\label{sec:normalizations}
Note that around any point $(x_0, y_0)$ in the interior of the support of the optimal coupling the $xy$-derivative of the cost function is non-singular, hence we can normalize the cost function to be equal to the quadratic cost. More precisely, by an affine change of coordinates we can modify $c_\delta$ in such a way that $\widetilde{c}_{\delta}(\widetilde{x}_0, \widetilde{y}_0) = -\widetilde{x}_0 \cdot \widetilde{y}_0$.

In the following we write $ \psi$ in place of $\psi_\delta$. Let $x_0 \in X$ and set $y_0 = T(x_0)$, i.e.\ 
\begin{align}\label{eq:x0y0-support}
	-\nabla\psi(x_0) + \nabla_x c_{\delta}(x_0, y_0) = 0.
\end{align}
 
Introduce 
\begin{align*}
	\overline{c}_{\delta}(x,y) &\coloneq c_{\delta}(x,y) - c_{\delta}(x, y_0) - c_{\delta}(x_0, y) + c_{\delta}(x_0, y_0)\\
	\overline{\psi}(x) &\coloneq \psi(x) - \psi(x_0) - c_{\delta}(x, y_0) + c_{\delta}(x_0, y_0) \\
	\overline{\psi}^{\overline{c}_{\delta}}(y) &\coloneq \psi^{c_{\delta}}(y) - \psi(x_0) - c_{\delta}(x_0, y) + c_{\delta}(x_0, y_0).
\end{align*}
Then $\overline{\psi}$ is $\overline{c}_{\delta}$-concave. %and $T(x) = \cdeltaexp_x(-\nabla\psi(x)) = \cbardeltaexp_x(-\nabla\overline{\psi}(x))$.
Note that 
\begin{enumerate}[label=(\roman*)]
	\item $\overline{c}_{\delta}(x_0,y_0) = 0$, $\nabla_x \overline{c}_{\delta}(x,y_0)=0$ for all $x\in\RR^d$, $\nabla_y \overline{c}_{\delta}(x_0,y) = 0$ for all $y\in \RR^d$, $D_x^2 \overline{c}_{\delta}(x,y_0) = 0$ for all $x\in\RR^d$, $D_y^2 \overline{c}_{\delta}(x_0,y) = 0$ for all $y\in\RR^d$, and $D_{xy}^2 \overline{c}_{\delta}(x,y) = D_{xy}^2 c_{\delta}(x,y)$, $D_{yx}^2 \overline{c}_{\delta}(x,y) = D_{yx}^2 c_{\delta}(x,y)$ for all $(x,y) \in \RR^d\times \RR^d$. In particular, the matrix $M\coloneq D_{yx} \overline{c}_{\delta}(x_0, y_0) = D_{yx} c_{\delta}(x_0, y_0) = D_{yx} c_0(x_0, y_0)$ is non-degenerate by \eqref{eq:coulomb-nondegenerate}. 
	\item $\overline{\psi}(x_0) = 0$, $\nabla\overline{\psi}(x_0) = \nabla{\psi}(x_0) - \nabla_xc_{\delta}(x_0, y_0)=0$ by \eqref{eq:x0y0-support}.
	\item Since $x_0$ is an Alexandrov point of $\psi$, there exists a symmetric matrix $D^2\psi(x_0)$ such that 
		\begin{align*}
			\nabla{\psi}(x) = \nabla{\psi}(x_0) + D^2\psi(x_0)(x-x_0) + o(|x-x_0|) \quad \text{as} \quad x\to x_0.
		\end{align*} 
		Moreover, as $(x_0,y_0)$ is a maximizer of $\psi_\delta(x)+\phi_\delta(y)-c_\delta(x,y)$ (see \eqref{eq:support-bound-upper} and \eqref{eq:support-bound-lower}), the second order optimality condition with respect to $x$ gives that the matrix $A \coloneq D^2\psi(x_0)-D_{xx}^2 c_{\delta}(x_0, y_0)$ is negative semi-definite. 
		Hence, $\overline{\psi}$ is twice differentiable at $x_0$ with 
		\begin{align*}
			\nabla{\overline{\psi}}(x) 
		%	&= \left(\nabla{\psi}(x_0) - \nabla_xc_{\delta}(x_0, y_0)\right) + \left(D^2\psi(x_0)-D_x^2 c_{\delta}(x_0, y_0)\right)(x-x_0) + o(|x-x_0|)  \\
			&= A(x-x_0) + o(|x-x_0|)  \quad \text{as} \quad x\to x_0.
		\end{align*}
\end{enumerate}

\subsection{Smallness of the non-dimensional local (Euclidean) transport energy}\label{sec:smallness-energy}
We now use that the map $(p,x) \mapsto \Cexp{c_{\delta}}_x(p)$ is a well-defined $\mathcal{C}^1$ function, since the equation $p + \nabla_x c_{\delta}(x,y)=0$ has a unique solution $y=y(x,p)$ for any $p\in\RR^d$ as long as $|x-y|> \frac{2}{3}\delta$. It follows that the map $T(x)=\Cexp{c_{\delta}}_x(-\nabla\psi(x))$ is differentiable at $x_0$ with 
\begin{align}\label{eq:T-expansion}
	T(x) &= T(x_0) + DT(x_0) (x-x_0) + o(|x-x_0|)\nonumber \\
	&= y_0 + M^{-1}A (x-x_0) + o(|x-x_0|)  \quad\qquad \text{as} \quad x\to x_0,
\end{align}
where we used that, from differentiating $-\nabla\psi(x)+\nabla_x c_\delta(x,T(x))=0$ with respect to $x$, 
$$
 -D^2\psi(x) + D_{xx}c_\delta(x,T(x)) + D_{yx}c_\delta(x,T(x))DT(x) =0 
$$
and hence, by solving for $DT(x)$ and setting $x=x_0$,
$$
 DT(x_0) = D_{yx}c_{\delta}(x_0,y_0)^{-1} (D^2\psi(x_0) - D_{xx}^2 c_{\delta}(x_0, y_0)) = M^{-1}A.
$$
In particular, we obtain that 
\begin{align}\label{eq:small-energy-1}
	&\frac{1}{R^{d+2}} \int_{B_R(x_0)\times \RR^d} |y-y_0 - M^{-1}A(x-x_0)|^2\,\gamma(\dd x \dd y) \nonumber \\
	&= \frac{1}{R^{d+2}} \int_{B_R(x_0)} |T(x)-y_0 - M^{-1}A(x-x_0)|^2\,\rho_0(\dd x) 
	\stackrel{\eqref{eq:T-expansion}}{=} \frac{o(R^2)}{R^2} \to 0 \quad \text{as} \quad R\downarrow 0.
\end{align}

Since $A \leq 0$, we may perform the change of coordinates $\widetilde{x} \coloneq (-A)^{\frac{1}{2}} x$, $\widetilde{y} \coloneq (-A)^{-\frac{1}{2}} M y$, in order to achieve that $-\widetilde{x}\cdot\widetilde{y} = - x\cdot M y$, and define 
\begin{align*}
	\widetilde{T}(\widetilde{x}) &\coloneq (-A)^{-\frac{1}{2}} M T((-A)^{-\frac{1}{2}}\widetilde{x}),\\
	\widetilde{c}_{\delta}(\widetilde{x}, \widetilde{y}) &\coloneq \overline{c}_{\delta}((-A)^{-\frac{1}{2}} \widetilde{x}, M^{-1}(-A)^{\frac{1}{2}}\widetilde{y}),\\
	\widetilde{\psi}(\widetilde{x}) &\coloneq \overline{\psi}((-A)^{-\frac{1}{2}}\widetilde{x}),\\
	\widetilde{\rho_0}(\widetilde{x}) &\coloneq \det((-A)^{-\frac{1}{2}}) \rho_0((-A)^{-\frac{1}{2}}\widetilde{x}),\\
	\widetilde{\rho_1}(\widetilde{y}) &\coloneq |\det(M^{-1}(-A)^{\frac{1}{2}})|\, \rho_1(M^{-1}(-A)^{\frac{1}{2}}\widetilde{y}).
\end{align*}
Then $\widetilde{\psi}$ is $\widetilde{c}_{\delta}$-concave and the cost $\widetilde{c}_{\delta}$ satisfies $D^2_{\widetilde{y}\widetilde{x}}\widetilde{c}_{\delta}(\widetilde{x}_0, \widetilde{y}_0) = -\mathbb{I}$, where $\widetilde{x}_0 = (-A)^{\frac{1}{2}} x_0$, $\widetilde{y}_0 = (-A)^{-\frac{1}{2}} M y_0$.

Denoting by $\widetilde{\gamma}$ the push-forward of the $c_{\delta}$-optimal (i.e.\ $c_0$-optimal) coupling $\gamma$ between $\mu$ and $\nu$ under the change of coordinates $(x,y)\mapsto(\widetilde{x},\widetilde{y})$, it follows that $\widetilde{\gamma}$ is optimal between $\widetilde{\mu}$ and $\widetilde{\nu}$ with respect to the cost $\widetilde{c}_{\delta}$. Indeed,  $\widetilde{\gamma}$ has the correct marginals and (following the  argumentation in \cite{OPR21})
\begin{align*}
	&\int \widetilde{c}_{\delta}(\widetilde{x}, \widetilde{y})\, \widetilde{\gamma}(\dd\widetilde{x}\dd\widetilde{y}) \\
	&= \int c_{\delta}(x,y) \,\gamma(\dd x \dd y) - \int c_{\delta}(x, y_0)\,\mu(\dd x) - \int c_{\delta}(x_0, y)\,\nu(\dd y) + c_{\delta}(x_0, y_0) \\
	&= \inf_{\gamma' \in \Pi(\mu, \nu)} \int c_{\delta}\,\dd \gamma' - \int c_{\delta}(x, y_0)\,\mu(\dd x) - \int c_{\delta}(x_0, y)\,\nu(\dd y) + c_{\delta}(x_0, y_0) \\
	&= \inf_{\gamma' \in \Pi(\mu, \nu)} \int \overline{c}_{\delta}\,\dd \gamma'
	= \inf_{\widetilde{\gamma}' \in \Pi(\widetilde{\mu}, \widetilde{\nu})} \int \widetilde{c}_{\delta}\,\dd \widetilde{\gamma}'.
\end{align*} 
This implies that $\widetilde{T}$ is the $\widetilde{c}_{\delta}$-optimal map from $\widetilde{\mu}(\dd \widetilde{x})\coloneq\widetilde{\rho_0}(\widetilde{x})\dd \widetilde{x}$ to $\widetilde{\nu}(\dd \widetilde{y})\coloneq \widetilde{\rho_1}(\widetilde{y})\dd \widetilde{y}$. In particular, since 
\begin{align*}
    \det D\widetilde{T}(\widetilde{x}) = \frac{\widetilde{\rho_0}(\widetilde{x})}{\widetilde{\rho_1}(\widetilde{T}(\widetilde{x}))},
\end{align*}
it follows that $\widetilde{\rho_0}(\widetilde{x_0}) = \widetilde{\rho_1}(\widetilde{y_0})$, so that by dividing $\widetilde{\rho_0}$ and $\widetilde{\rho_1}$ by the same constant we may assume without loss of generality that $\widetilde{\rho_0}(\widetilde{x_0}) = \widetilde{\rho_1}(\widetilde{y_0})= 1$.

Finally, with this change of coordinates, \eqref{eq:small-energy-1} turns into
\begin{align}\label{eq:small-energy-2}
	\EE_R^+ \coloneq \frac{1}{R^{d+2}} \int_{B_R(\widetilde{x}_0)\times \RR^d} |\widetilde{y} - \widetilde{y}_0 - (\widetilde{x}-\widetilde{x}_0)|^2\,\widetilde{\gamma}(\dd \widetilde{x}\dd\widetilde{y}) \to 0 \quad \text{as} \quad R\downarrow 0.
\end{align}
\subsection{\texorpdfstring{Local smallness of the Hölder semi-norms of the data and $yx$-derivative of the cost}{Local smallness of the Hölder semi-norms of the data and yx-derivative of the cost}}\label{sec:holder}
Let us denote for functions $f: \RR^d \to \RR$ and $k:\RR^d \times \RR^d \to \RR^{d\times d}$ the local Hölder semi-norms by
\begin{align*}
	[f]_{\alpha, B_R(x_0)} &\coloneq \sup_{x\neq x'\in B_R(x_0)} \frac{|f(x) - f(x')|}{|x-x'|^{\alpha}} \\
	[k]_{\alpha, B_R(x_0) \times B_R(y_0)} &\coloneq \sup_{(x,y)\neq (x',y') \in B_R(x_0) \times B_R(y_0)} \frac{|k(x,y)-k(x', y')|}{|x-x'|^{\alpha} + |y-y'|^{\alpha}}.
\end{align*}

The cost function $\widetilde{c}_{\delta}\in\mathcal{C}^{2,1}(\RR^d\times\RR^d)$ has a.e.\ bounded derivatives and is twisted on the support of $\widetilde{\gamma}$, where $D^2_{\widetilde{y}\widetilde{x}}\widetilde{c}_{\delta}$ is non-degenerate. 
Hence
\begin{align*}
	\mathcal{K}_R \coloneq R^{2\alpha} [D^2_{\widetilde{y}\widetilde{x}}\widetilde{c}_{\delta}]^2_{\alpha, B_R(\widetilde{x}_0) \times B_R(\widetilde{y}_0)} \stackrel{R\to0}{\longrightarrow} 0.
\end{align*}
In particular, for any $a\geq 1$ there holds 
\begin{align}\label{eq:xy-c0-to-0}
	\|D^2_{\widetilde{y}\widetilde{x}}\widetilde{c}_{\delta} + \mathbb{I}\|_{\mathcal{C}^0(B_{R}(\widetilde{x}_0)\times B_{a R}(\widetilde{y}_0))}^2 
	&= \sup_{\widetilde{x} \in B_{R}(\widetilde{x}_0), \widetilde{y} \in B_{a R}(\widetilde{y}_0)} |D^2_{\widetilde{y}\widetilde{x}}\widetilde{c}_{\delta}(x,y) - D^2_{\widetilde{y}\widetilde{x}}\widetilde{c}_{\delta}(\widetilde{x}_0, \widetilde{y}_0)|^2 \nonumber \\
	&\leq (1 + a^{\alpha})^2 R^{2\alpha} [D_{\widetilde{y}\widetilde{x}} \widetilde{c}_{\delta}]_{\alpha, B_{R}(\widetilde{x}_0)\times B_{a R}(\widetilde{y}_0)} \nonumber \\
	&\leq \left(\frac{1 + a^{\alpha}}{a^{\alpha}}\right)^2 \mathcal{K}_{a R} \stackrel{R\to 0}{\longrightarrow} 0.
\end{align}

The probability densities $\widetilde{\rho}_0$ and $\widetilde{\rho}_1$ are both $\mathcal{C}^{0,\alpha}$ functions, hence
\begin{align}
	\mathcal{D}_R \lesssim R^{2\alpha} [\widetilde{\rho}_0]^2_{\alpha, B_R(\widetilde{x}_0)} + R^{2\alpha} [\widetilde{\rho}_1]^2_{\alpha, B_R(\widetilde{y}_0)} \stackrel{R\to0}{\longrightarrow} 0.
\end{align}

\subsection{Notation}
To avoid heavy notation, let us from now on drop the tilde and assume that $c_{\delta}\in C^{2,1}(\RR^d\times \RR^d)$ is a bounded cost function with bounded (first and second) derivatives, which satisfies the twist condition on an open set $J \subset \RR^d \times \RR^d$ containing the support of the $c$-optimal coupling $\gamma$ between two measures $\mu$ and $\nu$ that have $\alpha$-Hölder continuous densities $\rho_0$ and $\rho_1$ that are bounded and strictly positive on their support. 
Without loss of generality we may assume that 
\begin{enumerate}[label=(\roman*)]
	\item $c_\delta(x_0, y_0) = 0$, $\nabla_x c_{\delta}(x, y_0) = 0$ for all $x\in\RR^d$, $\nabla_y c_{\delta}(x_0, y) = 0$ for all $y\in\RR^d$, $D^2_{xx} c_{\delta}(x, y_0) = 0$ for all $x\in\RR^d$, $D^2_{yy} c_{\delta}(x_0, y) = 0$ for all $y\in\RR^d$, and $D_{yx}c(x_0, y_0)=-\mathbb{I}$.
	\item $\rho_0(x_0) = \rho_1(y_0) = 1$. 
\end{enumerate}

\subsection{Bound on the local displacement}
We now proceed to bounding the local displacement around the point $x_0$. As a first step we prove there exists a scale $R_0>0$ below which smallness of $\EE^+$ and $\DD$ implies that if $x\in \pi_1(\support\gamma)$ with $|x-x_0| < R$, then $T(x)\in B_{\Lambda_0 R}$ for some $\Lambda_0<\infty$. The corresponding proof in \cite[Lemma~2.1]{OPR21} has to be adapted in two aspects: 
\begin{enumerate}[label=(\roman*)]
	\item The, in general, non-compactness of the supports of $\mu$ and $\nu$; here we will use the global boundedness of $c_{\delta}$ together with its derivatives.
	\item The fact that $c_{\delta}$ is twisted only on the support of $\gamma$.
\end{enumerate}

\begin{lemma}[Qualitative $L^{\infty}$ bound on the displacement of optimal maps]\label{lem:displacement-qualitative} 
	Let $\gamma\in\Pi(\mu, \nu)$ be a coupling with $c_{\delta}$-monotone support and let $(x_0, y_0) \in \support\gamma$. 
	
	There exist $\Lambda_0 <\infty$, $R_0>0$, and $\varepsilon_0>0$ such that for all $R\leq R_0$ for which 
	\begin{align}\label{eq:displacement-qualitative-assumption}
		\EE_{6R}^+ + \DD_{6R} \leq \varepsilon_0,
	\end{align}
	we have the inclusion 
	\begin{align*}
		(B_{5R}(x_0)\times \RR^d) \cap \supp \gamma \subseteq  B_{5R}(x_0) \times B_{\Lambda_0 R}(y_0).
	\end{align*}
\end{lemma}
For details, we refer the interested reader to Appendix~\ref{sec:displacement}.

Lemma~\ref{lem:displacement-qualitative} implies that, while under Coulomb transport a point $x_0 \in \support\mu$ may be mapped to a very distant point $y_0 \in \support \nu$, a small enough neighborhood $B_{5R}(x_0)$ of $x_0$ is transported into a small neighborhood of $y_0$ (small in the sense that it is of order $R$), as long as the local Euclidean transport energy $\EE_{6R}^+$ is sufficiently small and the marginals are locally sufficiently close to uniform measure in the sense that $\DD_{6R} \ll 1$. 

Appealing to the regularity of $c_{\delta}$ we can now strengthen the (more qualitative) bound on the displacement of Lemma~\ref{lem:displacement-qualitative}:

\begin{lemma}[$L^2\!-\!L^\infty$ estimate on optimal maps]\label{lem:displacement-quantitative} 
	Let $\gamma\in\Pi(\mu, \nu)$ be a coupling with $c_{\delta}$-monotone support and let $(x_0, y_0) \in \support\gamma$. 
	
	There exist $\varepsilon\in(0,1)$ and $M<\infty$ such that for all $\Lambda<\infty$ and for all $R>0$ such that 
	\begin{enumerate}[label=(\roman*)]
		\item $(B_{5R}(x_0) \times \RR^d) \cap \support\gamma \subset B_{5R}(x_0) \times B_{\Lambda R}(y_0)$
		\item $\EE_{6R}^+ + \DD_{6R} \leq \varepsilon$
		\item $\|D_{yx}c_{\delta}+\mathbb{I}\|_{C^0(B_{5R}(x_0)\times B_{\Lambda R}(y_0))} \leq \varepsilon$
	\end{enumerate}
	there holds
	\begin{align*}
		(x,y) \in (B_{4R}(x_0) \times \RR^d) \cap \support\gamma \quad \implies \quad |x-x_0-(y-y_0)|\leq M R \left( \EE_{6R}^+ + \DD_{6R} \right)^{\frac{1}{d+2}}.
	\end{align*}
\end{lemma}
In particular, if $\EE_{6R}^+ + \DD_{6R}$ is small enough, then $x \in B_{4R}(x_0)$ will imply that $y \in B_{5R}(y_0)$. Note that assumption (iii) is always fulfilled for $R$ small enough in view of \eqref{eq:xy-c0-to-0}. 

The core ingredient in the proof of Lemma~\ref{lem:displacement-quantitative} is $c_{\delta}$-monotonicity of the support of $\gamma$, combined with the closeness of $c_{\delta}$ to the quadratic cost in the sense of  assumption (iii). This is expressed in the following estimate, see \cite[Equation (2.5)]{OPR21} for details: let $(x,y) \in (B_{4R}(x_0) \times \RR^d) \cap \support\gamma$; then for all $(x',y') \in (B_{5R}(x_0) \times \RR^d) \cap \support \gamma$ there holds
\begin{align}\label{eq:c-monotonicity-displacement-proof}
%	&(x-x_0 - (y-y_0)) \cdot(x-x') \nonumber\\
%	&\leq 3 |x-x'|^2 + |x'-x_0 - (y'-y_0)|^2 \nonumber \\
%	&\quad + \|D_{yx}c_{\delta}+\mathbb{I}\|_{C^0(B_{5R}(x_0)\times B_{\Lambda R}(y_0))} |x-x'| |x-x_0 - (y-y_0)|.
0\geq -(x-x')\cdot(y-y') - \|D_{yx}c_{\delta}+\mathbb{I}\|_{C^0(B_{5R}(x_0)\times B_{\Lambda R}(y_0))} |x-x'| |y-y'|
\end{align}

Moreover, we have that  
\begin{align}\label{eq:linfty-inverse-transport}
	(\RR^d \times B_{2R}(y_0)) \cap \support\gamma \subseteq B_{4R}(x_0) \times B_{2R}(y_0).
\end{align}
Indeed, assume for contradiction that there exists $(x,y) \in (B_{4R}(x_0)^c \times B_{2R}(y_0))\cap \support\gamma$. %This point satisfies $|y-y_0 - (x-x_0)| \geq 2R$. 
Then consider the point $x'$ uniquely characterized by 
\begin{enumerate}
	\item $x'-x_0$ is a convex combination of $x-x_0$ and $y-y_0$,
	\item $|x'-x_0|=\frac{5}{2}R$.
\end{enumerate}
Let $y'\in \RR^d$ be such that $(x',y') \in \support\gamma$. 
Since $x'\in B_{\frac{5}{2}R}(x_0)$, it follows that \[|x'-x_0 - (y'-y_0)| \leq M \varepsilon^{\frac{1}{d+2}} R.\]

Then $x-x' = - \alpha (y-y_0 - (x-x_0))$ and $|x-x'| = \alpha |y-y_0 - (x-x_0)|$ for some $\alpha\in(0,1)$. 
Now by \eqref{eq:c-monotonicity-displacement-proof}, using that $\|D_{yx}c_{\delta}+\mathbb{I}\|_{C^0(B_{5R}(x_0)\times B_{\Lambda R}(y_0))} \leq \varepsilon < 1$ and writing $y-y' = y-y_0 - (x'-x_0) + x'-x_0 - (y'-y_0)$, we get
\begin{align*}
	0&\leq (x-x')\cdot (y-y_0 - (x'-x_0)) + (x-x')\cdot (x'-x_0 - (y'-y_0)) \\ 
	&\quad + \varepsilon |x-x'| |y-y_0 - (x'-x_0)| + \varepsilon |x-x'| |x'-x_0 - (y'-y_0)| \\
	& = -\alpha |y-y_0 - (x'-x_0)|^2 + (x-x')\cdot (x'-x_0 - (y'-y_0)) \\ 
	&\quad + \varepsilon |x-x'| |y-y_0 - (x'-x_0)| + \varepsilon |x-x'| |x'-x_0 - (y'-y_0)| \\
	&\leq |x-x'| \Big( (-1+\varepsilon) |y-y_0 - (x'-x_0)| + (1+\varepsilon) |x'-x_0 - (y'-y_0)|\Big).
\end{align*}
Since $|y-y_0 - (x'-x_0)| \geq |x'-x_0| - |y-y_0| \geq \frac{5}{2}R - 2R = \frac{R}{2}$, together with Lemma~\ref{lem:displacement-quantitative} we obtain that 
\begin{align*}
	0 &\leq |x-x'| R \left( (-1+\varepsilon) \frac{1}{2} + (1+\varepsilon) M \varepsilon^{\frac{1}{d+2}} \right),
\end{align*}
which is impossible for $\epsilon$ small enough. 

% Let us for a given point $(x_0, y_0) \in \RR^d \times \RR^d$ and radius $R>0$ denote by   
% \begin{align*}
% 	\#_{R}(x_0, y_0) \coloneq (B_{R}(x_0) \times \RR^d) \cup (\RR^d \times B_{R}(y_0)) 
% \end{align*}
% the ``infinite cross'' centered at $(x_0, y_0)$ of width $R$. 

As a consequence of \eqref{eq:linfty-inverse-transport}, for $\varepsilon$ small enough we may estimate the two-sided local Euclidean transport energy around the point $(x_0, y_0)$ by 
\begin{align}\label{eq:two-sided-energy-bound}
	\EE_{2R} &\coloneq \frac{1}{(2R)^{d+2}} \int_{\#_{2R}(x_0,y_0)} |y-x|^2\,\gamma(\dd x \dd y) \nonumber\\
	&\leq \frac{1}{(2R)^{d+2}} \int_{B_{2R}(x_0) \times \RR^d} |y-x|^2\,\gamma(\dd x \dd y) + \frac{1}{(2R)^{d+2}} \int_{\RR^d \times B_{2R}(y_0)} |y-x|^2\,\gamma(\dd x \dd y) \nonumber \\
	&\leq \frac{1}{(2R)^{d+2}} \int_{B_{2R}(x_0) \times \RR^d} |y-x|^2\,\gamma(\dd x \dd y) + \frac{1}{(2R)^{d+2}} \int_{B_{4R}(x_0) \times B_{2R}(y_0)} |y-x|^2\,\gamma(\dd x \dd y) \nonumber \\
	&\leq 2 \cdot 3^{d+2} \EE_{6R}^+.
\end{align}

\subsection{Almost-minimality with respect to Euclidean cost}
We now prove that the optimal coupling for the Coulomb cost function is locally almost-optimal around $(x_0,y_0)$ for the Euclidean cost function. To this end, note that the coupling $\gamma_R \coloneq \gamma \lfloor_{\#_{2R}(x_0, y_0)}$ is $c_0$-optimal between its own marginals, i.e.\ between the measures $\mu_R$ and $\nu_R$ defined via 
\begin{align*}
	\mu_R(A) &\coloneq \gamma_R (A\times \RR^d) = \gamma((A\times \RR^d) \cap \#_{2R}(x_0, y_0)),\\
	\nu_R(B) &\coloneq \gamma_R (\RR^d \times B) = \gamma((\RR^d \times B) \cap \#_{2R}(x_0, y_0))
\end{align*}
for all Borel subsets $A,B \subseteq \RR^d$. 
In particular, $\support \mu_R\subseteq B_{4R}(x_0)$, $\support \nu_R\subseteq B_{4R}(y_0)$, hence $\support\gamma_R \subseteq B_{4R}(x_0) \times  B_{4R}(y_0)$. Moreover, $\mu_R \leq \mu$, $\nu_R \leq \nu$, and $\mu_R = \mu$ on $B_{2R}(x_0)$, $\nu_R = \nu$ on $B_{2R}(y_0)$. 

We may therefore apply \cite[Proposition~1.10]{OPR21} to conclude that 
\begin{align}\label{eq:almost-minimality}
	 \int \tfrac{1}{2} |y-x|^2\,\gamma_R(\dd x \dd y) \leq \int \tfrac{1}{2} |y-x|^2\,\gamma_R'(\dd x \dd y) + R^{d+2} \Delta_R
\end{align}
for any $\gamma_R' \in \Pi(\mu_R, \nu_R)$, where 
\begin{align*}
	\Delta_R \coloneq C \|D_{yx}c_{\delta} + \mathbb{I}\|_{\mathcal{C}^0(B_{2R}(x_0)\times B_{2R}(y_0))} \EE_{2R}^{\frac{1}{2}}
\end{align*}
for some constant $C<\infty$ depending only on the dimension $d$. 

Once the above almost-minimality of $\gamma_R$ with respect to the Euclidean cost has been established, our cost $c_\delta$ (whose properties differ slightly from those studied in \cite{OPR21}) is no longer needed, except for the decay properties of the quantity $\|D_{yx}c_{\delta} + \mathbb{I}\|_{\mathcal{C}^0(B_{2R}(x_0)\times B_{2R}(y_0))}$ in the remainder $\Delta_R$ in \eqref{eq:almost-minimality}. 
Regarding the latter, we have that 
\begin{align*}
    \|D_{yx}c_{\delta} + \mathbb{I}\|_{\mathcal{C}^0(B_{2R}(x_0)\times B_{2R}(y_0))} \leq (2R)^{\alpha} [D_{yx}c_{\delta}]_{\alpha, B_{2R}(x_0)\times B_{2R}(y_0)} \to 0 \quad \text{as } R \to 0, 
\end{align*}
since on $B_{2R}(x_0) \times B_{2R}(y_0)$ the cost function $c_{\delta} = c$ is $\mathcal{C}^{\infty}$; see Section~\ref{sec:holder} above.
The remaining arguments in \cite{OPR21}, in particular the harmonic approximation \cite[Theorem 1.13]{OPR21} and the one-step improvement \cite[Proposition 1.16]{OPR21} at the heart of the Campanato iteration in \cite[Proof of Theorem 1.1, Step 3-6]{OPR21} rely solely on arguments and constructions based on the properties of $\gamma_R$ with respect to the cost $|y-x|^2$ and no longer involve special properties of our cost $c_\delta$.

The $\varepsilon$-regularity result of \cite{OPR21} therefore applies and gives that the optimal transport map $T$ is $\mathcal{C}^{1,\alpha}$ in $B_{\frac{R}{2}}(x_0)$.

\subsection{Conclusion}

Undoing the coordinate changes of Sections~\ref{sec:normalizations} and \ref{sec:smallness-energy}, it follows that the optimal transport map $T$ is a $\mathcal{C}^{1,\alpha}$ diffeomorphism between a neighborhood $U_0$ of $x_0$ and the neighborhood $T(U_0)$ of $y_0$. In particular, $U_0 \times T(U_0) \subseteq X \times Y$, so that $X \times Y$ and therefore $X$ and $Y$ are open. 
Together with \eqref{eq:inverse-relation2} it follows that $T$ is a global $\mathcal{C}^{1,\alpha}$ diffeomorphism between $X$ and $Y$. \hfill $\blacksquare$
%%%%%%%%%%%%%%%%%%%%%%%%%%%%%%%%%%%%%%%%%%%%%%%%%%
\section*{Acknowledgments}
This work was partially carried out while TR was a visiting research fellow at TU München, funded by 
\emph{Deutsche Forschungsgemeinschaft (DFG -- German Research Foundation) -- Project-ID 195170736 -- TRR109}, 
and a group leader at the Max Planck Institute for Mathematics in the Sciences Leipzig.

TR gratefully acknowledges partial support by the U.S.\ National Science Foundation under award DMS-2453121.

The authors would like to thank the Department of Mathematics at TUM and the MPI in Leipzig for the support and kind hospitality. 

\appendix
\hypertarget{appendix}{}
\section{A bound on the displacement}\label{sec:displacement}
In this section, we provide the details of the proof of Lemma~\ref{lem:displacement-qualitative}.

We also need the following lemma from \cite{OPR21}, which states that if the local transport energy and the local Hölder seminorm around a point in the support of an optimal plan are small, we can always find a second nearby point in the support which lies near any prescribed direction: 
\begin{lemma}[Lemma A.1 in \cite{OPR21}]\label{lem:cone}
	Let $\gamma\in \Pi(\mu, \nu)$. There exists $\varepsilon>0$ such that the following holds: if  $R>0$ is such that 
	\begin{align*}
		\EE_{6R}^+ + R^{2\alpha}[\rho_0]_{\alpha, B_{6R}(x_0)}^2 \leq \varepsilon,
	\end{align*}
	then for any $x\in B_{5R}(x_0)$ and $e\in\mathbb{S}^{d-1}$ there holds 
	\begin{align*}
		(S_R(x,e) \times B_{7R}) \cap \support\gamma \neq \emptyset,
	\end{align*}
	where $S_R(x,e)\coloneq C(x,e) \cap (B_R(x) \setminus B_{\frac{R}{2}}(x))$ is the intersection of the annulus $B_R(x) \setminus B_{\frac{R}{2}}(x)$ with the spherical cone $C(x,e)$ of opening angle $\frac{\pi}{2}$ with apex at $x$ and axis along $e$. 
\end{lemma}

\begin{proof}[Proof of Lemma~\ref{lem:displacement-qualitative}]
	\begin{enumerate}[label=\textsc{Step \arabic*},leftmargin=0pt,labelsep=*,itemindent=*]
	\item \label{step:grad-boundedness} (Use of $c$-monotonicity of $\supp\gamma$).
	Let $\varepsilon>0$ be such that Lemma~\ref{lem:cone} holds and assume that for some radius $R>0$ we have $\EE^+_{6R} + \DD_{6R} \leq \varepsilon$.

	We claim that there exists a constant $\lambda<\infty$, depending only on $\|c_{\delta}\|_{\mathcal{C}^2(\RR^d\times\RR^d)}$, such that 
	\begin{align}\label{eq:nabla-c-bound}
		\nabla_x c_{\delta}(x,y) \in B_{\lambda R} \quad \text{for all} \quad (x,y) \in (B_{5R}(x_0) \times \RR^d) \cap \supp \gamma. 
	\end{align}
	To show this, we use the $c_{\delta}$-monotonicity of $\supp \gamma$, i.e.\
	\begin{align}\label{eq:tilde-monotonicity}
		c_{\delta}(x,y) - c_{\delta}(x',y) \leq c_{\delta}(x,y') - c_{\delta}(x',y') 
		\quad \text{for all} \quad (x,y),(x',y') \in \supp \gamma.
	\end{align}
	Without loss of generality we may assume $\nabla_x c_{\delta}(x,y) \neq 0$. 
	
	With $x_t := tx + (1-t)x'$ we can write
	\begin{align*}
		c_{\delta}(x,y) - c_{\delta}(x',y) 
		&= \int_0^1 \nabla_x c_{\delta}(x_t, y)\,\dd t \cdot (x-x') \\
		&= \nabla_x c_{\delta}(x, y) \cdot (x-x') + \int_0^1 (\nabla_x c_{\delta}(x_t, y) - \nabla_x c_{\delta}(x, y)) \,\dd t \cdot (x-x'),
	\end{align*}
	and, using that $\nabla_x c_{\delta}(x_0,y_0) = 0$,
	\begin{align*}
		c_{\delta}(x,y') - c_{\delta}(x',y') 
		&= (\nabla_x c_{\delta}(x_0, y') - \nabla_x c_{\delta}(x_0,y_0)) \cdot (x-x') \\
		&\quad + \int_0^1 (\nabla_x c_{\delta}(x_t, y') - \nabla_x c_{\delta}(x_0, y')) \,\dd t \cdot (x-x').
	\end{align*} 
	Inserting these two identities into inequality \eqref{eq:tilde-monotonicity} gives
	\begin{align*}
		\nabla_x c_{\delta}(x, y) \cdot (x-x') 
		&\leq \int_0^1 |\nabla_x c_{\delta}(x_t, y) - \nabla_x c_{\delta}(x, y)| \,\dd t \, |x-x'|  \\
        &\quad + |\nabla_x c_{\delta}(x_0, y') - \nabla_x c_{\delta}(x_0,y_0)| \, |x-x'| \\
		&\quad +\int_0^1 |\nabla_x c_{\delta}(x_t, y') - \nabla_x c_{\delta}(x_0, y')| \,\dd t \, |x-x'|.
	\end{align*}
	Using the boundedness of $\|c_{\delta}\|_{\mathcal{C}^2(\RR^d\times \RR^d)}$ we  estimate this expression further by 
	\begin{align}
		\nabla_x c_{\delta}(x, y) \cdot (x-x') 
		&\leq \|\nabla_{xx} c_{\delta}\|_{\mathcal{C}^0(\RR^d\times \RR^d)} \left( \int_0^1 |x_t - x|\,\dd t + \int_0^1 |x_t|\,\dd t \right) |x-x'| \nonumber\\
		&\quad + \|\nabla_{xy} c_{\delta} \|_{\mathcal{C}^0(\RR^d\times \RR^d)} |y'| \, |x-x'| \nonumber\\
		&\lesssim \|c\|_{\mathcal{C}^2(\RR^d\times \RR^d)} \left( |x'| + |x| + |y'| \right) |x-x'|.\label{eq:nablax-estimate}
	\end{align}

	Now by Lemma~\ref{lem:cone}, given $x \in B_{5R}(x_0)$, we have $(S_{R}(x,e) \times B_{7R}(y_0)) \cap \supp\gamma \neq \emptyset$ for any direction $e\in \mathbb{S}^{d-1}$. Hence, letting $e = \frac{\nabla_x c_{\delta}(x,y)}{|\nabla_x c_{\delta}(x,y)|}$, we can find a point $(x',y') \in (S_{R}(x,e) \times B_{7R}(y_0))\cap \supp\gamma$. Since the opening angle of $S_R(x,e)$ is $\frac{\pi}{2}$, we have 
	\begin{align*}
		\nabla_x c_{\delta}(x,y) \cdot (x-x') 
		= |\nabla_{x}c_{\delta}(x,y)| |x-x'| e \cdot \frac{x-x'}{|x-x'|}
		\gtrsim |\nabla_{x}c_{\delta}(x,y)| |x-x'|.
	\end{align*}
	It follows with \eqref{eq:nablax-estimate} that there exists $\lambda<\infty$ such that 
	\begin{align*}
		|\nabla_x c_{\delta}(x,y)| \lesssim \|c\|_{\mathcal{C}^2(\RR^d\times\RR^d)} \left(|x'| + |x| + |y'| \right) \leq \lambda R.
	\end{align*}
		
	\medskip
	\item \label{step:implicitfunction} (Use of twistedness of $c$ on $\support\gamma$). 
    Let $J \subset \RR^d \times \RR^d$ be the open set on which $c_{\delta}\in C^{2,1}(\RR^d\times \RR^d)$ satisfies the twist condition and that contains $\supp\gamma$. 
    
	We claim that there exist $R_0>0$ and $\Lambda_0<\infty$ such that for all $R\leq R_0$ 
%	for which $|\nabla_x \widetilde{c}(x,y)| \leq \lambda R$ 
	and $x\in B_{5R}(x_0)$, we have that 
	\begin{align*}
		B_{\lambda R}\cap -\nabla_x c_{\delta}(x,\pi_Y(J))\subseteq -\nabla_x c_{\delta}(x,B_{\Lambda_0 R}(y_0)\cap \pi_Y(J))
	\end{align*}
	
	Indeed, since $c_{\delta}$ is twisted on the open set $J \subset \RR^d\times\RR^d$, for any $x\in B_{5R}(x_0)$, the map $-\nabla_x c_{\delta}(x,\cdot)$ is one-to-one on the open set $\pi_Y(J)$.
	Hence, the map 
	\begin{align*}
		F_x : -\nabla_x c_{\delta}(x,\pi_Y(J)) \to \pi_Y(J), \quad p \mapsto \left[ - \nabla_x c_{\delta}(x,\cdot) \right]^{-1}(p)
	\end{align*}
	is well-defined and a $\mathcal{C}^{1,1}$-diffeomorphism, so that in particular
	\begin{align*}
		F_x(p) = F_x(p) + DF_x(0)p + O_x(|p|^2).
	\end{align*}
	Using that $-\nabla_x c_{\delta}(x,y_0) = 0$, which translates into $F_x(0) = y_0$, and that $DF_x(0) = D_{yx}c_{\delta}(x, y_0)^{-1}$ is non-degenerate, we obtain 
	\begin{align}\label{eq:F-bound}
		|F_x(p) - y_0| \leq |D_{yx}c_{\delta}^2(x, y_0)^{-1}p| + O_x(|p|^2).
	\end{align}
	Appealing to Remark~\ref{rem:inverse-Dxyc}, we see that since $|x_0-y_0|>\delta$ for $(x_0, y_0) \in \support\gamma$, choosing $R$ small enough so that $|x-y_0| \geq \frac{3}{4} \delta$ for all $x \in B_{5R}(x_0)$, i.e.\ $R\leq \frac{\delta}{20}$, the supremum over all $x\in B_{5R}(x_0)$ on the right hand side of \eqref{eq:F-bound} is bounded. It follows that there exist a radius $R_0>0$ and a constant $\Lambda_0<\infty$ such that
	\begin{align*}
		\lambda |F_x(p)-y_0| \leq \Lambda_0 |v| 
		\quad \text{for all} \quad 
		x\in B_{5R}(x_0) \quad \text{and} \quad 
		|p| \leq \lambda R_0,
	\end{align*}
	which we may reformulate as
	\begin{align*}
		F_x(B_{\lambda R}\cap -\nabla_x c_{\delta}(x,\pi_Y(J))) \subseteq B_{\Lambda_0 R}(y_0)\cap \pi_Y(J), 
		\intertext{i.e.} 
		B_{\lambda R}\cap -\nabla_x c_{\delta}(x,\pi_Y(J)) \subseteq -\nabla_x c_{\delta}(x,B_{\Lambda_0 R}(y_0)\cap \pi_Y(J))
	\end{align*}
	for all $R\leq R_0$ and $x\in B_{5R}(x_0)$.
	
	\medskip
	\item \label{step:qualitative-bound} (Conclusion). 
	If $(x,y) \in (B_{5R}(x_0) \times \RR^d) \cap \supp\gamma$, then we claim that $|y-y_0| \leq \Lambda_0 R$ for $R\leq R_0$.
	
	Indeed, by \ref{step:grad-boundedness} we have $\nabla_x c_{\delta}(x,y) \in B_{\lambda R}\cap -\nabla_x c_{\delta}(x,\pi_Y(J))$. 
	Since $$B_{\lambda R} \cap -\nabla_x c_{\delta}(x,\pi_Y(J)) \subseteq -\nabla_x c_{\delta}(x,B_{\Lambda_0 R}(y_0) \cap \pi_Y(J))$$ by \ref{step:implicitfunction},
	injectivity of $y \mapsto -\nabla_x c_{\delta}(x,y)$  on $\pi_Y(J)$ implies that we must have $y \in B_{\Lambda_0 R}(y_0)$. \qedhere
\end{enumerate}
\end{proof}

\subsection*{Data availability}
No datasets were generated or analyzed during the current study.

\subsection*{Conflicts of interest}
The authors have no competing interests to declare that are relevant to the content of this article.

%%%%%%%%%%%%%%%%%%%%%%%%%%%%%%%%%%%%%%%%%%%%%%%%%%

\bigskip
%%%%%%%%%%%%%%%%%%%%%%%%%%%%%%%%%%%%%%%%%%%%%%%%%%
\end{document}